\documentclass[a4paper,11pt,oneside]{article}
\usepackage[latin1]{inputenc}
\usepackage{amssymb, amsmath}
\usepackage{moreverb}
\usepackage{vmargin}
\usepackage{fancyhdr}
\usepackage{amsfonts}
\usepackage{bm}
\usepackage{fancybox}
\usepackage{amsthm}
\usepackage{amsmath}
\newtheorem{theorem}{Theorem}[section]

\newtheorem{lemma}[theorem]{Lemma}

\newtheorem{proposition}[theorem]{Proposition}

\newtheorem{definition}[theorem]{Definition}

\usepackage[
  linktocpage,
  urlbordercolor={1 0.5 0.125},
  filebordercolor={1 0.5 0.125},
  linkbordercolor={0.25 1 0.25},
  citebordercolor={0.25 1 0.25},
  pdfborderstyle={/S/U/W 1}, 
  breaklinks]{hyperref}  
  
\begin{document}
\ \begin{center}
\begin{LARGE}
\bf {A model of individual clustering\\
with vanishing diffusion} \\
\end{LARGE}
\vspace{0.5cm}
Elissar Nasreddine\\
\vspace{0.2cm}
 \begin{small}
\textit{Institut de Math\'ematiques de Toulouse, Universit\'e de Toulouse,}\\
 \textit{F--31062 Toulouse cedex 9, France}\\
 \vspace{0.1cm}
 e-mail: elissar.nasreddine@math.univ-toulouse.fr\\
 \today
\end{small}
\end{center}
\textbf{Abstract}. We consider a model of individual clustering with two specific reproduction rates and small diffusion parameter in one space dimension. It consists of a drift-diffusion equation for the population density coupled to an elliptic equation for the velocity of individuals. We prove the convergence (in suitable topologies) of the solution of the problem to the unique solution of the limit transport problem, as the diffusion coefficient tends to zero.
\section{Introduction}
In \cite{models}, a model for the dispersal of individuals with an additional aggregation mechanism is proposed. More precisely, the population density $u(t,x)$ at location $x\in \Omega$ where $\Omega$ is an open bounded domain of $\mathbb{R}^N$, $1\leq N\leq3$, and time $t>0$ solves the convection-diffusion equation
\begin{equation}\label{in1.4}
\partial_t u= \delta\ \Delta u-\nabla\cdot(u\ \bm{\omega})+r\ u\ E(u),
\end{equation}
where $\delta>0$, $r\geq0$ and $E$ is the net rate reproduction per individual .
This equation is coupled to an elliptic equation for the velocity $\bm{\omega}$ which is assumed to be in the direction of increasing $E(u)$, say, of the form $\lambda\nabla E(u)$ with $\lambda >0$. The evolution of the velocity $\bm{\omega}$ is described by 
\begin{equation}\label{in2.4}
-\varepsilon\ \Delta \bm{\omega}+\bm{\omega}=\lambda\ \nabla E(u),
\end{equation}
where $\varepsilon>0$ and $\varepsilon \ \Delta\bm{\omega}$ is simply to smooth out any sharp local variation in $ \nabla E(u)$ so that $\bm{\omega}$ represents a local average of the velocity $\lambda\ \nabla E(u)$.\\
We supplement \eqref{in1.4} and \eqref{in2.4} with no-flux boundary conditions
\begin{equation}\label{in3.4}
n\cdot\nabla u=n\cdot\bm{\omega}=0, \ x\in \partial\Omega, t\geq0,
\end{equation}
as suggested in \cite{models} where $n$ is the outward normal of $\partial \Omega$. In addition, in dimension $2$ or $3$, we impose the following additional condition given in \cite{asimplified,unprobleme,quelques} to guarantee the well-posedness of the elliptic system \eqref{in2.4}
\begin{equation}\label{in4.4}
\partial_n\bm{\omega}\times n=0,\ x\in \partial\Omega, t\geq0.
\end{equation}
As usual, $v\times \omega$ is the number $v_1\ \omega_1+v_2\ \omega_2$ if $N=2$ and the vector field
 $(v_2\ \omega_3-v_3\ \omega_2, -v_1\ \omega_3+v_3\ \omega_1,v_1\ \omega_2-v_2\ \omega_1)$ if $N=3$.\\
 
 We are interested here in the case where the aggregation mechanism is dominant, that is, the diffusivity $\delta$ is small. For biological models, this can change dramatically the dynamical behaviour of the solutions, and might generate finite time blow-up such as for the Keller-Segel system, see \cite{finitetime} for instance. Nevertheless, there are situation for which the small diffusivity limit is somehow ``stable", including some models from semiconductor physics, see  Markowich and Szmolyan \cite{asystem}, and for the Keller-Segel system with volume-filling effect, see \cite{thekeller, existenceof}. There, the authors prove the convergence, in the small diffusivity limit, of the solutions of the parabolic systems to  weak entropy solutions of the corresponding hyperbolic systems. A related field of research which is currently very active is the analysis of the so-called aggregation equation $\partial_tu+\mathrm{div}(K(u))=0$ where $K$ is a nonlocal linear operator, see \cite{localand,globalintime}, and the references therein.\\
  Taking the case of small diffusivity as a motivation, we will study the system \eqref{in1.4}, \eqref{in2.4}, \eqref{in3.4} and \eqref{in4.4}, in the limit of vanishing diffusivity $\delta$. More precisely,
given a sufficiently smooth function $E$, parameters $\delta\in (0,1)$, $\varepsilon >0$ and $r\geq 0$, our aim in this paper is to investigate the limit $\delta\rightarrow 0$ of the following one dimensional system
\begin{equation}
\label{in5.4}
\left\{
\begin{array}{llll}
\displaystyle \partial_t u_\delta&=& \delta\ \partial_x^2 u_\delta-\partial_x (u_\delta\ \varphi_\delta)+r\ u_\delta\ E(u_\delta),& x\in (-1,1), t>0 \\
\displaystyle -\varepsilon\ \partial_x^2\varphi_\delta+\varphi_\delta&=& \partial_x E(u_\delta),& x\in (-1,1), t>0 \\
\displaystyle \partial_xu_\delta(t,\pm1)&=&\varphi_\delta(t,\pm1)=0,& t>0\\
\displaystyle u_\delta(0,x)&=&u_0(x),& x\in (-1,1).
\end{array}
\right.
\end{equation}
where $E$ has two specific forms suggested in \cite{models}, namely
\begin{equation}
\label{in7.4}
E(u)=1-u
\end{equation}
or 
\begin{equation}
\label{in8.4}
E(u)=(1-u)(u-a),\ \ \mathrm{for\ some}\ a\in (0,1).
\end{equation}
 
Given $u_0\in W^{1,2}(-1,1)$, the existence and uniqueness of a global solution of \eqref{in5.4} have been shown in \cite{well}, and the purpose of this paper is to prove that $(u_\delta, \varphi_\delta)$ converges to a solution of the nonlocal transport problem
\begin{equation}\label{in6.4}
\left\{
\begin{array}{llll}
\displaystyle \partial_t u&=&-\partial_x \left(u\ \varphi\right)+r\ u\ E(u),& \ \ x\in (-1,1), t>0,\\
\displaystyle-\varepsilon\  \partial^2_x \varphi+\varphi&=& \partial_x E(u),& \ \ x\in (-1,1), t>0,\\
\displaystyle \varphi(t,\pm1)&=&0,& t>0,\\
\displaystyle u(0,x)&=&u_0(x),& x\in (-1,1),
\end{array}
\right.
\end{equation}
 as the diffusion coefficient $\delta$ approaches zero. This leads, in a natural way, to the existence of a smooth solution of \eqref{in6.4}, the uniqueness being established in Proposition \ref{pr1.4}.\\
 
 Our paper is organized as follows. In Section $2$, we state the main results and focus on the two specific forms of $E$ suggested in \cite{models}: the ``bistable case" \eqref{in8.4}, see Theorem \ref{th1.4}, and the  ``monostable case" \eqref{in7.4}, see Theorem \ref{th2.4}. In Section $3$, we recall some results of existence and uniqueness of a global solution of \eqref{in5.4} obtained in \cite{well}. Section $4$ is devoted to the uniqueness issue of smooth solutions of the transport problem \eqref{in6.4}. In Section 5, we focus on the bistable case \eqref{in8.4}. We derive a priori estimates on $(u_\delta, \varphi_\delta)$, which are uniformly valid in $\delta$, and particularly we derive a lower bound for $\partial_x\varphi_\delta$  and an $L^\infty(W^{1,1})$ estimate on $u_\delta$ which leads to an $L^\infty(W^{2,2})$ bound on $u_\delta$. these estimates imply, by a compactness argument, the existence of accumulation points of any sequence $(u_\delta, \varphi_\delta)_\delta$. Thanks to Section 3, we conclude that the limit of $(u_\delta, \varphi_\delta)_\delta$ is unique, and the whole family $(u_\delta, \varphi_\delta)$ converges to the unique solution of \eqref{in6.4} with $E(u)=(1-u)(u-a)$. In Section 6, we analyse the limit $\delta\rightarrow 0$ of \eqref{in5.4} in the monostable case \eqref{in7.4}. This analysis is quite similar to that of the previous case, except for the first estimate.
\section{Main results}
Throughout this paper, and unless otherwise stated, we assume that
$$\delta\in (0,1),\  \varepsilon>0,\  r\geq 0.$$
In \cite{well}, the global existence and uniqueness of smooth solution of \eqref{in5.4} are shown when $E(u)$ has the structure \eqref{in7.4} or \eqref{in8.4}. Our first result gives the limit $\delta \rightarrow  0$ in \eqref{in5.4} in the bistable case, that is when $E(u)=(1-u)(u-a)$ for some $a\in(0,1)$.
\begin{theorem}\label{th1.4}
Assume that $u_0$ is a nonnegative function in $W^{1,2}(-1,1)$ and $E(u)=(1-u)(u-a)$ for some $a\in(0,1)$. For $\delta>0$, let $u_\delta$ be the global nonnegative solution to \eqref{in5.4} given by Theorem \ref{th3.4} below. Then, for all $T>0$ 
\begin{equation*}
\lim\limits_{\delta \to 0}||u_{\delta}(t)- u(t)||_C=0\ \ \mathrm{for\ all}\ t\in(0,T),
\end{equation*}
where $u\in C\left([0,T]; L^2(-1,1)\right)\cap L^\infty\left((0,T); W^{1,2}(-1,1)\right)$ is the unique smooth solution of the corresponding transport system 
\begin{equation} \label{we1.4}\partial_t u=-\partial_x \left(u\ \varphi\right)+r\ u\ (1-u)(u-a),\ \ x\in (-1,1), t>0, \end{equation}
\begin{equation}\label{we3.4}-\varepsilon\  \partial^2_x \varphi+\varphi=(-2\ u+a+1)\ \partial_xu,\ \ x\in (-1,1), t>0\end{equation}
with boundary and initial conditions
\begin{equation}
\label{we4.4}
\varphi(t,\pm1)=0,\ \mathrm{for\ all}\ t>0,\ \mathrm{and}\ u(0,x)=u_0(x),\ x\in(-1,1).
\end{equation}
\end{theorem}
As a consequence of \eqref{we4.4} no boundary conditions for $u$ are needed.\\
The proof of the previous theorem is performed by deriving estimates which are uniformly valid for $0<\delta<1$. This proof starts with the suitable cancellation of the coupling terms in two equations which gives an estimate for $u_\delta$ in $L^\infty (L^2)$ and for $\varphi_\delta$ in $L^2(W^{1,2})$. Then we derive a lower bound for $\partial_x\varphi_\delta$ and an $L^\infty(W^{1,1})$ bound on $u_\delta$ which leads to an $L^\infty(W^{1,2})$ bound on $u_\delta$. We will, by a compactness argument, show the convergence of $u_\delta$ to the smooth solution of the transport system \eqref{we1.4}, \eqref{we3.4} and \eqref{we4.4}.\\

Next, we turn to the monostable case, that is when $E(u)=1-u$, and we study the limit $\delta\rightarrow 0$. 
\begin{theorem}\label{th2.4}
Assume that $u_0$ is a nonnegative function in $W^{1,2}(-1,1)$ and $E(u)=(1-u)$. For $\delta>0$ let $u_\delta$ be the global nonnegative solution of \eqref{in5.4} given by Theorem \ref{th3.4} below. Then, for all $T>0$ 
\begin{equation*}
\lim\limits_{\delta \to 0}||u_{\delta}(t)- u(t)||_C=0\ \ \mathrm{for\ all}\ t\in(0,T),
\end{equation*}
where $u\in C\left([0,T]; L^2(-1,1)\right)\cap L^\infty\left((0,T); W^{1,2}(-1,1)\right)$ is the unique smooth solution of the following transport system,
\begin{equation}\label{we5.4} \partial_t u=-\partial_x \left(u\ \varphi\right)+r\ u\ (1-u),\ \ x\in (-1,1), t>0,\end{equation}
\begin{equation}\label{we6.4}-\varepsilon\  \partial^2_x \varphi+\varphi=- \partial_xu,\ \ x\in (-1,1), t>0,\end{equation}
with boundary and initial conditions
\begin{equation}\label{we7.4}
\varphi(t,\pm1)=0,\ \mathrm{for\ all}\ t>0,\ \mathrm{and}\ u(0,x)=u_0(x),\ x\in(-1,1).
\end{equation}
\end{theorem}
The proof of Theorem \ref{th2.4} follows the same lines as that of Theorem \ref{th1.4}. As in the bistable case, there is a cancellation between the two equations but it only gives an $L^\infty(L \log L)$ bound on $u_\delta$ and an $L^2(W^{1,2})$ bound on $\varphi_\delta$.
\section{Well-Posedness of \eqref{in5.4}}
We first recall the notion of solution to \eqref{in5.4} to be used in this paper.
\begin{definition}\label{de1.4}
Let $T>0$, $E\in C^2(\mathbb{R})$, and an initial condition $u_0\in W^{1, 2}(-1,1)$ . For  $0<\delta<1$, a strong solution to \eqref{in5.4} on $[0, T)$ is a function
$$u_\delta \in C \left( [0,T), W^{1,2}(-1,1)\right)\cap C\left( (0,T), W^{2,2}(-1,1)\right),$$ such that
\begin{equation*}
\left\{
\begin{array}{llll}
\displaystyle\partial_t u_\delta&=& \delta\ \partial_x^2 u_\delta-\partial_x(u_\delta\ \varphi_\delta)+r\ u_\delta\ E(u_\delta),& \mathrm{a.e.\ in }\ [0,T)\times(-1,1)\\
\displaystyle u_\delta(0,x)&=&u_0(x),& \mathrm{a.e.\ in }\ (-1,1)\\
\displaystyle \partial_xu_\delta(t,\pm1)&=&0,& \mathrm{a.e.\ on }\ [0,T),
\end{array}
\right.
\end{equation*}
where, for all $t\in [0,T)$, $\varphi_\delta(t)$ is the unique solution in $W^{2,2}(-1,1)$ of
\begin{equation*}
\left\{
\begin{array}{llll}
\displaystyle-\varepsilon\partial_x^2\varphi_\delta(t)+\varphi_\delta(t)&=& \partial_x E(u_\delta(t))& \mathrm{a.e.\ in}\ (-1,1)\\
\displaystyle\varphi_\delta(t,\pm1)&=&0&
\end{array}
\right.
\end{equation*}
\end{definition}
We now recall the global existence theorem which is proved in \cite{well}, where $E(u)$ is given by \eqref{in7.4} or \eqref{in8.4}.
\begin{theorem} \label{th3.4}Assume that $u_0$ is a nonnegative function in $W^{1,2}(-1,1)$,\\ and $E(u)=(1-u)(u-a)$ for some $a\in (0,1)$ or $E(u)=1-u$. Then \eqref{in5.4} has a unique global nonnegative solution $u$ in the sense of Definition \ref{de1.4}.
\end{theorem}
\section{Uniqueness}
In this section we prove the uniqueness of the solution of \eqref{in6.4}. Let us first give the definition of the strong solution of \eqref{in6.4}.
\begin{definition}
\label{de2.4}
Let $T>0$, $E\in C^2(\mathbb{R})$, and an initial condition $u_0\in W^{1,2}(-1,1)$ . A strong solution on $[0, T)$ to the transport system \eqref{in6.4} is a function
$$u \in C \left( [0,T), L^2(-1,1)\right)\cap C\left( (0,T), W^{1,2}(-1,1)\right),$$ such that
\begin{equation*}
\left\{
\begin{array}{llll}
\displaystyle\partial_t u&=&-\partial_x(u\ \varphi)+r\ u\ E(u),& \mathrm{a.e.\ in }\ [0,T)\times(-1,1)\\
\displaystyle u(0,x)&=&u_0(x),& \mathrm{a.e.\ in }\ (-1,1),
\end{array}
\right.
\end{equation*}
where, for all $t\in [0,T)$, $\varphi(t)$ is the unique solution in $W^{2,2}(-1,1)$ of
\begin{equation*}
\left\{
\begin{array}{llll}
\displaystyle-\varepsilon\partial_x^2 \varphi(t)+\varphi(t)&=& \partial_x E(u(t))& \mathrm{a.e.\ in}\ (-1,1)\\
\displaystyle \varphi(t,\pm1)&=&0&.
\end{array}
\right.
\end{equation*}
\end{definition}
 The main result is contained in
\begin{proposition}\label{pr1.4}
Assume that $u_0$ is a nonnegative function in $W^{1,2}(-1,1)$ and $E\in C^2(\mathbb{R}).$ Then for all $T>0$, there exists at most one solution $u$ of \eqref{in6.4} in the sense of Definition \ref{de2.4}, such that 
\begin{equation}
\label{eq53.4}
u\in L^\infty\left((0,T), W^{1,1}(-1,1)\right),\ \ \mathrm{and}\ \ \varphi\in L^\infty\left((0,T); W^{1,\infty}(-1,1)\right).
\end{equation}
\end{proposition}
\begin{proof}
Let us assume that there exist two different solutions $u_1$ and $u_2$ to \eqref{in6.4} corresponding to the same initial conditions, and fix $T>0$. We put
$$(u, \varphi)=(u_1-u_2, \varphi_1-\varphi_2), \ \mathrm{in}\ [0,T]\times (-1,1).$$ Then $(u, \varphi)$ satisfies
\begin{equation}\label{eq51.4}
\left\{
\begin{array}{llll}
\displaystyle \partial_tu&=&-\partial_x(u\ \varphi_1)-\partial_x(u_2\ \varphi)+r\ u_1\ E(u_1)-r\ u_2\ E(u_2),&  \mathrm{in}\ (0,T)\times(-1,1)\\
\displaystyle -\varepsilon \partial_x^2 \varphi+\varphi&=&\ E'(u_1)\ \partial_xu_1-E'(u_2)\ \partial_xu_2,& \mathrm{in}\ (0,T)\times(-1,1)\\
\displaystyle \varphi(t,\pm1)&=&0& \mathrm{on}\ (0,T)\\
\displaystyle u(0,x)&=&0,& \mathrm{in}\ (-1,1).
\end{array}
\right.
\end{equation}
We multiply the first equation in \eqref{eq51.4} by sign$(u)$, and integrate it by parts over $(-1,1)$ to obtain
\begin{eqnarray}
\frac{\mathrm{d}}{\mathrm{d}t}||u||_1&=& -\int_{-1}^1\varphi_1\ \partial_x|u|\ dx-\int_{-1}^1\partial_x\varphi_1\ |u|\ dx\nonumber\\
&-&\int_{-1}^1 \mathrm{sign (u)}\ \partial_x\varphi\ u_2\ dx-\int_{-1}^1 \varphi\ \partial_xu_2\ \mathrm{sign (u)}\ dx\nonumber\\
&+&r\ \int_{-1}^1 (u_1\ E(u_1)-u_2\ E(u_2)))\ \mathrm{sign (u)}\ dx\nonumber\\
&\leq& ||\partial_x\varphi||_1\ ||u_2||_\infty+||\partial_xu_2||_1\ ||\varphi||_\infty+r\ ||u_1\ E(u_1)-u_2\ E(u_2)||_1,\label{eq55.4}
\end{eqnarray}
since the first line in the right-hand side vanishes. Using the fact that $u_1$ and $u_2$ are bounded by \eqref{eq53.4} and the embedding of $W^{1,1}(-1,1)$ in $L^\infty(-1,1)$ we estimate
\begin{equation}\label{eq54.4}
||u_1\ E(u_1)-u_2\ E(u_2)||_1\leq C\ ||u||_1.
\end{equation}
Using \eqref{eq53.4}, \eqref{eq54.4}, and the continuous embedding of $W^{1,1}(-1,1)$ in $L^\infty(-1,1)$ , \eqref{eq55.4} becomes
\begin{equation}\label{eq56.4}
\frac{\mathrm{d}}{\mathrm{d}t}||u||_1\leq C\ ||\partial_x\varphi||_1+C\ ||\varphi||_\infty+C\ ||u||_1.
\end{equation}
To complete the proof of Proposition \ref{pr1.4}, it remains to estimate $||\partial_x\varphi||_1$ and $||\varphi||_\infty$.\\

For $x,y \in (-1,1)$, we integrate the second equation in \eqref{eq51.4} to obtain
\begin{eqnarray}
-\varepsilon\ \int_y^x \partial^2_x \varphi(z)\ dz+\int_y^x \varphi(z)\ dz&=&\int_y^x \left( \partial_x E(u_1)-\partial_x E(u_2)\right)\ dz\nonumber\\
-\varepsilon\ \left(\partial_x \varphi(x)-\partial_x \varphi(y)\right)+\int_y^x \varphi(z)\ dz&=&[E(u_1(x))-E(u_2(x))-E(u_1(y))+E(u_2(y))].\nonumber
\end{eqnarray}
Next we integrate the above equality with respect to $y$ over $(-1,1)$ to obtain
\begin{equation*}
\partial_x \varphi(x)=\frac{1}{2\ \varepsilon}\int_{-1}^1\int_y^x \varphi(z)\ dzdy-\frac{1}{\varepsilon}\  E(u_1(x))+\frac{1}{\varepsilon}E(u_2(x))+\frac{1}{2\ \varepsilon}\int_{-1}^1 (E(u_1(y))-E(u_2(y)))\ dy.
\end{equation*}
This gives
\begin{equation*}
|\partial_x \varphi(x)|\leq\frac{||\varphi||_1}{\varepsilon}+\frac{1}{\varepsilon}\  |E(u_1(x))-E(u_2(x))|+\frac{1}{2\ \varepsilon}\ ||E(u_1)-E(u_2)||_1.
\end{equation*}
Therefore
\begin{equation*}
||\partial_x \varphi||_1\leq 2\ \frac{||\varphi||_1}{\varepsilon}+\frac{1}{\varepsilon}\ ||E(u_1)-E(u_2)||_1+\frac{1}{\varepsilon}\ ||E(u_1)-E(u_2)||_1.
\end{equation*}
Since $u_1$ and $u_2$ are bounded and $E\in C^2(\mathbb{R})$ we obtain
\begin{equation}\label{eq58.7}
||\partial_x \varphi||_1\leq 2\ \frac{||\varphi||_1}{\varepsilon}+C\ \frac{2}{\varepsilon}\ ||u||_1.
\end{equation}
It remains to prove an $L^1$ estimate to $\varphi$. For that purpose, we define,
for $i=1,2$, the function $\psi_i\in L^\infty((0,T), W^{2,\infty}(-1,1))$ solution of
\begin{equation}
\label{eq52.4}
\left\{
\begin{array}{lll}
\displaystyle -\partial^2_x \psi_i(t,x)&=&\varphi_i(t,x),\  \mathrm{in}\ (-1,1)\\
\displaystyle \psi_i(t,\pm1)&=&0.
\end{array}
\right.
\end{equation}
We multiply the second equation in \eqref{eq51.4} by $\psi=\psi_1-\psi_2$ and integrate it over $(-1,1)$ to obtain
\begin{eqnarray}
||\varphi||_2^2+||\partial_x \psi||_2^2&=&\int_{-1}^1 \partial_x \left( E(u_1)-E(u_2)\right)\ \psi\ dx\nonumber\\
&\leq& ||E(u_1)-E(u_2)||_1\ ||\partial_x \psi||_\infty\leq C\ ||u||_1\ ||\partial_x\psi||_\infty.\nonumber
\end{eqnarray}
By the continuous embedding of $W^{1,2}(-1,1)$ in $L^\infty(-1,1)$ and by \eqref{eq52.4}, the previous inequality reads
\begin{equation*}
||\partial_x\psi||_{W^{1,2}}\leq C\ ||u||_1
\end{equation*}
and
\begin{equation}\label{eq57.4}
||\varphi||_1\leq C\ ||\varphi||_2=C\ ||\partial^2_x\psi||_{2}\leq C(||\partial_x\psi||_{2}+||\partial^2_x\psi||_{2})\leq C\ ||u||_1.
\end{equation}
Substituting \eqref{eq57.4} into \eqref{eq58.7}, and by the continuous embedding of $W^{1,1}(-1,1)$ in $L^\infty(-1,1)$ we obtain
\begin{equation}\label{eq59.7}
||\varphi||_\infty\leq C\ ||\partial_x \varphi||_1\leq C\ ||u||_1.
\end{equation}
Finally, we substitute \eqref{eq59.7} into \eqref{eq56.4} we obtain
\begin{equation}\label{eq60.4}
\frac{\mathrm{d}}{\mathrm{d}t}||u||_1\leq C\ ||u||_1+r\ C \ ||u||_1.
\end{equation}
Gronwall's inequality applied to inequality \eqref{eq60.4} implies that the two solutions are identical, which proves Proposition \ref{pr1.4}.
\end{proof}
\section{The bistable case: $E(u)=(1-u)(u-a)$}
Let $T>0$, the system \eqref{in5.4} now reads 
\begin{equation}
\label{bc4}
\left\{
\begin{array}{llll}
\displaystyle \partial_t u_\delta&=& \delta\ \partial^2_{x} u_\delta-\partial_x(u_\delta\ \varphi_\delta)+r\ u_\delta\ (u_\delta-a)(1-u_\delta),&\  \mathrm{in}\ (0,T)\times(-1,1) \\
\displaystyle -\varepsilon\ \partial^2_{x} \varphi_\delta+\varphi_\delta&=&(-2u_\delta+(a+1))\ \partial_x u_\delta,&\  \mathrm{in}\ (0,T)\times(-1,1)  \\
\displaystyle \partial_x u_\delta(t,\pm 1)&=&\varphi_\delta(t,\pm1)=0,&\  \mathrm{on}\ (0,T),\\
\displaystyle u_\delta(0,x)&=&u_0(x),&\  \mathrm{in}\ (-1,1) ,
\end{array}
\right.
\end{equation}
for some $a\in (0,1)$.\\
Thanks to Theorem \ref{th3.4}, \eqref{bc4} has a unique global nonnegative solution in the sense of the Definition \ref{de1.4}.\\
Integrating \eqref{bc4} over $(0,T)\times (-1,1)$ and using the nonnegativity of $u_\delta$, we first observe that
\begin{equation}\label{eq6.4}
||u_\delta(t)||_1\leq ||u_0||_1+2\ r\ (1-a)\ T,\ \ \ \mathrm{for\ all}\ t\in [0,T].
\end{equation}
\subsection{Estimates}

\begin{lemma}\label{le0.4} There is $C_1(T)>0$ independent of $\delta$ such that
\begin{equation}\label{u4}
\int_0^T \left( \frac{\varepsilon}{2}\ ||\partial_x\varphi_\delta||_2^2+||\varphi_\delta||_2^2+2\ \delta\ ||\partial_x u_\delta||_2^2\right)\ dt\leq C_1(T),\ \ \mathrm{for\ all}\ t\in [0,T]
\end{equation}
\begin{equation}\label{dxu4}
||u_\delta(t)||_2\leq C_1(T),\ \ \mathrm{for\ all}\ t\in [0,T],
\end{equation}
and
\begin{equation}\label{phi4}
\int_0^T ||\varphi_\delta||^2_{\infty}\ dt\leq C_1(T)\ \ \mathrm{for\ all}\ t\in [0,T].
\end{equation}
\end{lemma}

\begin{proof}
Multiplying the first equation in \eqref{bc4} by $2\ u_\delta$ and integrating it over $(-1,1)$, we obtain 
\begin{equation}\label{cub4}
\frac{\mathrm{d}}{\mathrm{d}t}\int_{-1}^1|u_\delta|^2\ dx=-2\ \delta\int_{-1}^1|\partial_x u_\delta|^2\ dx+2\ \int_{-1}^1\ u_\delta\ \varphi_\delta\ \partial_x u_\delta\ dx+2\ r\ \int_{-1}^1 u_\delta^2\ E(u_\delta)\ dx.
\end{equation}
Multiplying now the second equation in \eqref{bc4} by $\varphi_\delta$ and integrating it over $(-1,1)$ we obtain 
\begin{equation}\label{fi4}
\varepsilon\ \int_{-1}^1|\partial_x\  \varphi_\delta|^2\ dx+\int_{-1}^1|\varphi_\delta|^2\ dx=-2\ \int_{-1}^1 u_\delta\ \varphi_\delta\ \partial_x u_\delta\ dx+(a+1)\int_{-1}^1\partial_x u_\delta\ \varphi_\delta\ dx.
\end{equation}
At this point we notice that the cubic terms on the right hand side of \eqref{cub4} and \eqref{fi4} cancel one with the other, and summing \eqref{fi4} and \eqref{cub4} we obtain 
\begin{equation}
\frac{\mathrm{d}}{\mathrm{d}t}||u_\delta||^2_2+\varepsilon\ ||\partial_x\varphi_\delta||_2^2+||\varphi_\delta||^2_2+2\ \delta\ ||\partial_x u_\delta||_2^2=2\ r \int_{-1}^1 u_\delta^2\ E(u_\delta)\ dx+ (a+1)\int_{-1}^1 \partial_x u_\delta\ \varphi_\delta\ dx.
\end{equation}
We integrate by parts and use Cauchy-Schwarz inequality to obtain
\begin{equation*}
(a+1)\ \int_{-1}^1 \partial_x u_\delta\ \varphi_\delta\ dx=-(a+1)\int_{-1}^1 u_\delta\ \partial_x \varphi_\delta\ dx\leq\frac{(a+1)^2}{2\ \varepsilon}\ ||u_\delta||_2^2+\frac{\varepsilon}{2}\ ||\partial_x \varphi_\delta||_2^2.
\end{equation*}
On the other hand, $u_\delta^2\ E(u_\delta)\leq 0$ if $u_\delta\notin (a,1)$ so that
$$\int_{-1}^1 u_\delta^2\ E(u_\delta)\ dx\leq 2\ (1-a)$$
The previous inequalities give that
\begin{equation*}\label{df}\frac{\mathrm{d}}{\mathrm{d}t}||u_\delta||^2_2+\frac{\varepsilon}{2}\ ||\partial_x\varphi_\delta||_2^2+||\varphi_\delta||^2_2+2\ \delta\ ||\partial_x u_\delta||_2^2\leq\frac{(a+1)^2}{2\ \varepsilon}\ ||u_\delta||_2^2+ 4\ r\ (1-a).\end{equation*}
Therefore, by a time integration, there exists $C_1(T)$ such that \eqref{u4} and \eqref{dxu4} hold. By the continuous embedding of $W^{1,2}(-1,1)$ in $L^\infty(-1,1)$ we obtain \eqref{phi4}.
\end{proof}
\begin{lemma}\label{le20.4}
There is $C_2(T)>0$ independent of $\delta$ such that
\begin{equation}
\label{eq70.4}
||\varphi_\delta(t)||_2\leq C_2(T),\ \mathrm{for\ all }\ t\in[0,T].
\end{equation}
\end{lemma}
\begin{proof}
We define the function $\psi\in L^2\left(0,T; W^{2,2}(-1,1)\right)$ solution of
\begin{equation}
\label{eq72.4}
\left\{
\begin{array}{llll}
\displaystyle -\partial^2_x\psi_\delta&=&\varphi_\delta&\ \mathrm{in}\ (-1,1)\\
\displaystyle \psi_\delta(t,\pm1)&=&0,& \mathrm{in}\ [0,T).
\end{array}
\right.
\end{equation}
Multiplying the second equation in \eqref{bc4} by $\psi_\delta$ and integrating it over $(-1,1)$ we obtain
\begin{eqnarray}
-\varepsilon\ \int_{-1}^1 \varphi_\delta\ \partial_x^2\psi_\delta\ dx-\int_{-1}^1 \partial^2_x\psi_\delta\ \psi_\delta\ dx&=&- \int_{-1}^1  E(u_\delta)\ \partial_x\psi_\delta\ dx\nonumber\\
\varepsilon\ ||\partial_x^2\psi_\delta||_2^2+||\partial_x\psi_\delta||_2^2&\leq& ||E(u_\delta)||_1\ ||\partial_x \psi_\delta||_\infty.\nonumber
\end{eqnarray}
Using the embedding of $W^{1,2}(-1,1)$ in $L^\infty(-1,1)$ and the specific form \eqref{in8.4} of $E$ we obtain
\begin{equation}
||\partial_x\psi_\delta||_{W^{1,2}}\leq C\ ||E(u_\delta)||_1\leq C\ (1+||u_\delta||_2^2).
\end{equation}
Using \eqref{eq72.4} and \eqref{dxu4} the above inequation becomes
\begin{equation}
||\varphi_\delta||_2=||\partial^2_x\psi_\delta||_2\leq ||\partial_x\psi_\delta||_{W^{1,2}}\leq C.
\end{equation}
\end{proof}
\begin{lemma}\label{le1.4}
For $0<\delta<1$, there exists $C(T)>0$ independent of $\delta$ such that
\begin{equation}
\partial_x \varphi_\delta (x)\geq -4\ ||\varphi_\delta||_\infty-\frac{(a+1)^2}{4\ \varepsilon}-C(T).
\end{equation}
\end{lemma}
\begin{proof}
For $x,y \in (-1,1)$, we integrate the second equation in \eqref{bc4} to obtain
\begin{eqnarray}
-\varepsilon\ \int_y^x \partial^2_{x}\varphi_\delta(z)\ dz+\int_y^x \varphi_\delta(z)\ dz&=&\int_y^x \left( -2u_\delta+a+1\right)\ \partial_xu_\delta\ dz\nonumber\\
-\varepsilon\ \left(\partial_x \varphi_\delta(x)-\partial_x \varphi_\delta(y)\right)+\int_y^x \varphi_\delta(z)\ dz&=&-[u_\delta^2(x)-u_\delta^2(y)]+(a+1)\ (u_\delta(x)-u_\delta(y)).\nonumber
\end{eqnarray}
Next we integrate the above equality with respect to $y$ over $(-1,1)$ to obtain
\begin{equation*}
-2\ \varepsilon\ \partial_x \varphi_\delta(x)=-\int_{-1}^1\int_y^x \varphi_\delta(z)\ dzdy-2\ u_\delta^2(x)+2\ (a+1)\ u_\delta(x)+||u_\delta||_2^2-(a+1)\ ||u_\delta||_1.
\end{equation*}
Since $\frac{1}{\varepsilon}\ u_\delta^2-\frac{(a+1)}{\varepsilon} u_\delta\geq -\frac{(a+1)^2}{4\ \varepsilon},$ and
\begin{equation*}
\int_{-1}^1\int_y^x \varphi_\delta(z)\ dzdy \geq-\int_{-1}^1 \int_{-1}^1||\varphi_\delta||_\infty\ dzdy\geq -4\ ||\varphi_\delta||_\infty,
\end{equation*}
it follows from \eqref{dxu4} that
\begin{eqnarray}
\partial_x \varphi_\delta(x)&=&\frac{1}{2\ \varepsilon}\int_{-1}^1\int_y^x \varphi_\delta(z)\ dzdy+\frac{1}{\varepsilon}\ u_\delta^2(x)-\frac{a+1}{\varepsilon}\ u_\delta(x)-\frac{1}{2\ \varepsilon}||u_\delta||_2^2+\frac{a+1}{2\ \varepsilon}\ ||u_\delta||_1\nonumber\\
&\geq& -\frac{2}{\varepsilon} \ ||\varphi_\delta||_\infty-\frac{(a+1)^2}{4\ \varepsilon}-C(T).\nonumber
\end{eqnarray}
\end{proof}
We continue with estimates for the derivatives of $u_\delta$.
\begin{lemma}\label{le2.4}
There is $C_3(T)>0$ independent of $\delta$ such that
\begin{equation}\label{eq.4}
||\partial_x u_\delta(t)||_1\leq C_3(T)\ \ \mathrm{for\ all}\ t\in [0,T].
\end{equation}
\end{lemma}
\begin{proof}
We set $g_\delta=\partial_xu_\delta$ to simplify the notation, and differentiate the first equation in \eqref{bc4} with respect to $x$. This yields
\begin{equation}\label{eq1.4}
\partial_t g_\delta+\partial_{x}^2 (u_\delta\ \varphi_\delta)-r\ \partial_x\left(u_\delta\ (1-u_\delta)\ (u_\delta-a)\right)=\delta\ \partial^2_{x} g_\delta.
\end{equation}
We define an approximation of the sign function by $\sigma_\gamma(z)=\sigma(\frac{z}{\gamma})$, $0<\gamma\ll 1$, with $\sigma$ smooth and increasing, $\sigma(0)=0$, and $\sigma(z)=\mathrm{sign}\ z$ for $|z|>1$. Then, with $\mathrm{abs}_\gamma(z)=\int_0^z\sigma_\gamma(\xi)\ d\xi$, the convergence of $\mathrm{abs}_\gamma(z)$ to $|z|$ as $\gamma\rightarrow 0$ is uniform in $z\in \mathbb{R}$.\\
Multiplying \eqref{eq1.4} by $\sigma_\gamma(g_\delta)$ and integrating with respect to $x$ yields
\begin{eqnarray}
\int_{-1}^1 \sigma_\gamma(g_\delta)\ \partial_t g_\delta\ dx&+&\int_{-1}^1\sigma_\gamma(g_\delta)\ [\partial_{x}(g_\delta\ \varphi_\delta)+ g_\delta\ \partial_x\varphi_\delta+u_\delta\ \partial^2_{x}\varphi_\delta]\ dx\nonumber\\
&-&r\int_{-1}^1 \sigma_\gamma(g_\delta)\ (-3\ u_\delta^2+2\ (a+1)\ u_\delta-a)\ g_\delta\ dx\nonumber\\
&=&\delta\ \int_{-1}^1 \sigma_\gamma(g_\delta)\ \partial_{x}^2 g_\delta\ dx.\label{eq2.4}
\end{eqnarray}
Since 
\begin{equation*}
\int_{-1}^1\sigma_\gamma(g_\delta)\ \partial_{x}(g_\delta\ \varphi_\delta)\ dx=-\int_{-1}^1\sigma'_\gamma(g_\delta)\ \partial_xg_\delta\ g_\delta\ \varphi_\delta\ dx,
\end{equation*}
and
\begin{equation*}
\int_{-1}^1 \sigma_\gamma(g_\delta)\ \partial_{x}^2 g_\delta\ dx=-\int_{-1}^1 (\partial_x g_\delta)^2\ \sigma'_\gamma(g_\delta)\ dx,
\end{equation*}
we obtain
\begin{eqnarray}
\frac{\mathrm{d}}{\mathrm{d}t}\int_{-1}^1 \mathrm{abs}_\gamma(g_\delta)\ dx&-&\int_{-1}^1\sigma'_\gamma(g_\delta)\ \partial_xg_\delta\ g_\delta\ \varphi_\delta\ dx\nonumber\\
&+&\int_{-1}^1\sigma_\gamma(g_\delta)\ g_\delta\ \partial_x\varphi_\delta\ dx+\int_{-1}^1\sigma_\gamma(g_\delta)\ u_\delta\ \partial^2_{x}\varphi_\delta\ dx\nonumber\\
&-&r\int_{-1}^1 \sigma_\gamma(g_\delta)\ (-3\ u_\delta^2+2\ (a+1)\ u_\delta-a)\ g_\delta\ dx\nonumber\\
&=&-\delta\ \int_{-1}^1 (\partial_xg_\delta)^2\ \sigma'_\gamma(g_\delta)\ dx\leq 0.\label{eq3.4}
\end{eqnarray}
The function $f_\gamma(z)=\sigma_\gamma(z)\ z-\mathrm{abs}_\gamma (z)$ satisfies $f'_\gamma(z)=\sigma'_\gamma(z)\ z$ and converges to $0$ uniformly in $z\in \mathbb{R}$. We integrate the second term in \eqref{eq3.4} by parts and we use the second equation in \eqref{bc4} in the fourth one to obtain
\begin{eqnarray}
\frac{\mathrm{d}}{\mathrm{d}t}\int_{-1}^1 \mathrm{abs}_\gamma(g_\delta)\ dx&+&\int_{-1}^1 f_\gamma(g_\delta)\ \partial_x\varphi_\delta\ dx
+\int_{-1}^1\sigma_\gamma(g_\delta)\ g_\delta\ \partial_x\varphi_\delta\ dx\nonumber\\
&+&\frac{1}{\varepsilon}\int_{-1}^1\sigma_\gamma(g_\delta)\ u_\delta\ \left(2u_\delta-a-1\right)\ g_\delta\ dx\nonumber\\
&+&\frac{1}{\varepsilon}\int_{-1}^1\sigma_\gamma(g_\delta)\ u_\delta\ \varphi_\delta\  dx\nonumber\\
&\leq&r\int_{-1}^1 \sigma_\gamma(g_\delta)\ (-3\ u_\delta^2+2\ (a+1)\ u_\delta-a)\ g_\delta\ dx.\label{eq5.4}
\end{eqnarray}
Now,
$|g_\delta|\geq \sigma_\gamma(g_\delta)\ g_\delta\geq 0$ and
 $\left(2u_\delta-a-1\right)\geq 0$ if $u_\delta\geq \frac{a+1}{2}$, so that
\begin{eqnarray}
\int_{-1}^1\sigma_\gamma(g_\delta)\ u_\delta\ \left(2u_\delta-a-1\right)\ g_\delta\ dx&= &\int_{u_\delta\geq \frac{a+1}{2}}\sigma_\gamma(g_\delta)\ u_\delta\ \left(2u_\delta-a-1\right)\ g_\delta\ dx\nonumber\\
&+&2\ \int_{0<u_\delta\leq \frac{a+1}{2}} u^2_\delta \ \sigma_\gamma(g_\delta)\ g_\delta\ dx\nonumber\\
&-&(a+1)\int_{0<u_\delta\leq \frac{a+1}{2}} \sigma_\gamma(g_\delta)\ g_\delta\ (u_\delta)\ dx\nonumber\\
&\geq&-\frac{(a+1)^2}{2}\int_{0<u_\delta\leq \frac{a+1}{2}} |g_\delta| dx\nonumber
\end{eqnarray}
and thus
\begin{equation}\label{eq4.4}
\int_{-1}^1\sigma_\gamma(g_\delta)\ u_\delta\ \left(2u_\delta-a-1\right)\ g_\delta\ dx\geq -\frac{(a+1)^2}{2}\int_{-1}^1\ |g_\delta|\ dx.
\end{equation}
Also, since $-3\ u_\delta^2+2\ (a+1)\ u_\delta-a\leq \frac{(a+1)^2}{2}$ and $|g_\delta|\geq \sigma_\gamma(g_\delta)\ g_\delta\geq 0$,
\begin{equation}\label{eq90.4}r\int_{-1}^1 \sigma_\gamma(g_\delta)\ (-3\ u_\delta^2+2\ (a+1)\ u_\delta-a)\ g_\delta\ dx\leq \frac{r\ (a+1)^2}{2}\int_{-1}^1 |g_\delta|\ dx.\end{equation}
Passing to the limit $\gamma\rightarrow 0$ in \eqref{eq5.4} the first term on the right-hand side vanishes. 
And it follows from Lemma \ref{le1.4}, \eqref{eq4.4}, \eqref{eq90.4} and \eqref{eq6.4} that
\begin{eqnarray}
\frac{\mathrm{d}}{\mathrm{d}t}\int_{-1}^1 |g_\delta|\ dx&-&\left(\frac{2}{\varepsilon} \ ||\varphi_\delta||_\infty+\frac{(a+1)^2}{4\ \varepsilon}+C_2(T)\right)\ \int_{-1}^1 |g_\delta|\ dx-\frac{(a+1)^2}{2\ \varepsilon}\int_{-1}^1\ |g_\delta|\ dx\nonumber\\
&\leq&\frac{||u_\delta||_1}{\varepsilon}\ ||\varphi_\delta||_\infty+\frac{r\ (a+1)^2}{2}\int_{-1}^1 |g_\delta|\ dx\nonumber\\
&\leq&\frac{C(T)}{\varepsilon}\ ||\varphi_\delta||_\infty+\frac{r\ (a+1)^2}{2}\int_{-1}^1 |g_\delta|\ dx.\label{eq7.4}
\end{eqnarray}
Integrating \eqref{eq7.4} in time, and using \eqref{phi4} yield that there exists $C_3(T)$ such that \eqref{eq.4} holds.
\end{proof}
\begin{lemma}\label{le3.4}
There is $C_4(T)>0$ independent of $\delta$ such that
\begin{equation}\label{eq9.4}
||u_\delta(t)||_\infty\leq C_4(T)\ \ \mathrm{for\ all}\ t\in [0,T],
\end{equation}
and
\begin{equation}\label{eq63.4}
||\partial_x\varphi_\delta(t)||_\infty\leq C_4(T)\ \ \mathrm{for\ all}\ t\in [0,T].
\end{equation}
\end{lemma}
\begin{proof}
For all $T>0$, \eqref{eq6.4} and Lemma \ref{le2.4} guarantee that
\begin{equation}\label{eq8.4}
||u_\delta(t)||_{W^{1,1}}\leq C(T)\ \ \mathrm{for\ all}\ t\in [0,T].
\end{equation}
Therefore, the continuous embedding of $W^{1,1}(-1,1)$ in $L^\infty(-1,1)$ implies that \eqref{eq9.4} holds. On the other hand, Lemma \ref{le2.4}, Lemma \ref{le20.4}, \eqref{eq9.4} and the second equation in \eqref{bc4} ensure that $\partial_x \varphi_\delta(t)$ is bounded in $W^{1,1}(-1,1)$ and by the continuous embedding of $W^{1,1}(-1,1)$ in $L^\infty(-1,1)$ we get \eqref{eq63.4}.
\end{proof}
\begin{lemma}\label{le10.4}
There is $C_5(T)>0$ independent of $\delta$ such that
\begin{equation}\label{eq64.4}
||\partial_x u_\delta(t)||_2\leq C_5(T)\ \ \mathrm{for\ all}\ t\in [0,T].
\end{equation}
\end{lemma}
\begin{proof}
Coming back to \eqref{eq1.4} , $g_\delta=\partial_xu_\delta$ satisfies
\begin{equation}\label{eq65.4}
\partial_t g_\delta=-\partial_{x}g_\delta\ \varphi_\delta-2\ g_\delta\ \partial_x \varphi_\delta-u_\delta\ \partial_x^2\varphi_\delta+r\ \partial_x\left(u_\delta\ (1-u_\delta)\ (u_\delta-a)\right)+\delta\ \partial^2_{x} g_\delta.
\end{equation}
We multiply \eqref{eq65.4} by $g_\delta$ and integrate it over $(-1,1)$ to obtain
\begin{eqnarray}
\frac{1}{2}\frac{\mathrm{d}}{\mathrm{d}t} ||g_\delta||_2^2&=&-\int_{-1}^1 \partial_x g_\delta\ \varphi_\delta\ g_\delta\ dx-2\int_{-1}^1  |g_\delta|^2\ \partial_x\varphi_\delta\ dx-\int_{-1}^1  u_\delta\ \partial_x^2\varphi_\delta\ g_\delta\ dx\nonumber\\
&+&r\ \int_{-1}^1 \partial_x\left(u_\delta\ (1-u_\delta)\ (u_\delta-a)\right)\ g_\delta\ dx+\delta\ \int_{-1}^1\partial^2_{x} g_\delta\ g_\delta\ dx.\nonumber
\end{eqnarray}
We integrate by parts the last term of the right-hand side, and use the second equation in \eqref{bc4} to obtain
\begin{eqnarray}
\frac{1}{2}\frac{\mathrm{d}}{\mathrm{d}t} ||g_\delta||_2^2&=&-\int_{-1}^1 \partial_x (\frac{|g_\delta|^2}{2})\ \varphi_\delta\ dx-2\int_{-1}^1  |g_\delta|^2\ \partial_x\varphi_\delta\ dx+\frac{1}{\varepsilon}\int_{-1}^1  u_\delta\ (-\varphi_\delta+E'(u_\delta)\ g_\delta)\ g_\delta\ dx\nonumber\\
&+&r\ \int_{-1}^1 \left(-3\ u^2_\delta+2(a+1)u_\delta-a\right)\ |g_\delta|^2\ dx-\delta\ \int_{-1}^1|\partial_{x} g_\delta|^2\ dx.\nonumber
\end{eqnarray}
We integrate the first term in the right-hand side by parts and use H\"older inequality  to obtain
\begin{eqnarray}
\frac{1}{2}\frac{\mathrm{d}}{\mathrm{d}t} ||g_\delta||_2^2&=&\int_{-1}^1  \frac{|g_\delta|^2}{2}\ \partial_x\varphi_\delta\ dx-2\int_{-1}^1  |g_\delta|^2\ \partial_x\varphi_\delta\ dx-\frac{1}{\varepsilon}\int_{-1}^1  u_\delta\ \varphi_\delta\ g_\delta\ dx\nonumber\\
&+&\frac{1}{\varepsilon}\int_{-1}^1 u_\delta\ (-2u_\delta+a+1)\ |g_\delta|^2\ dx+r\ \int_{-1}^1 \left(-3\ u^2_\delta+2(a+1)u_\delta-a\right)\ |g_\delta|^2\ dx\nonumber\\
&\leq& \frac{3}{2}\ ||g_\delta||^2_2\ ||\partial_x\varphi_\delta||_\infty+\frac{1}{\varepsilon}\ ||u_\delta||_\infty\ ||\varphi_\delta||_2\ ||g_\delta||_2\nonumber\\
&+&\frac{1}{\varepsilon}\ ||u_\delta\ (-2u_\delta+a+1)||_\infty\ ||g_\delta||_2^2+r\ ||-3\ u^2_\delta+2(a+1)u_\delta-a||_\infty\ ||g_ \delta||_2^2.\nonumber
\end{eqnarray}
Using \eqref{eq9.4}, \eqref{eq63.4} and Young inequality we obtain
\begin{equation}\label{eq66.4}
\frac{\mathrm{d}}{\mathrm{d}t} ||g_\delta||_2^2\leq C\ ||g_\delta||_2^2+C\ ||\varphi_\delta||_2^2,
\end{equation}
it follows from \eqref{u4} after integration that \eqref{eq64.4} holds.
\end{proof}
\begin{lemma}\label{le4.4}
There is $C_6(T)>0$ independent of $\delta$ such that
\begin{equation}\label{eq71.4}
||\partial_t u_\delta(t)||^2_{(W^{1,2})'}\leq C_6(T)\ \ \mathrm{for\ all}\ t\in [0,T],
\end{equation}
where $(W^{1,2})'$ denotes the dual space of $W^{1,2}$.
\end{lemma}
\begin{proof}
Consider $\psi \in W^{1,2}(-1,1)$ and $t\in (0, T)$. We have by the first equation in \eqref{bc4}
\begin{eqnarray}
&&\left\lvert\int_{-1}^1 \partial_t u_\delta\ \psi\ dx\right\lvert\nonumber\\
&=&\left\lvert\int_{-1}^1\left[ \partial_x\left(\delta\ \partial_x u_\delta-u_\delta\ \varphi_\delta\right)+r\ u_\delta\ E(u_\delta)\right]\ \psi\ dx\right\lvert\nonumber\\
&=& \left\lvert\int_{-1}^1 \left( -\delta\ \partial_x u_\delta\ \partial_x\psi+u_\delta\ \varphi_\delta\ \partial_x\psi +r\ u_\delta\ (1-u_\delta)\ (u_\delta-a)\ \psi\right)\ dx\right\lvert\nonumber\\
&\leq& \delta\ ||\partial_x\psi||_2\ ||\partial_x u_\delta||_2+||\partial_x\psi||_2\ \ ||u_\delta||_\infty\ ||\varphi_\delta||_2+r\ ||u_\delta\ (1-u_\delta)(u_\delta-a)||_2\ ||\psi||_2\nonumber.
\end{eqnarray} 
Using \eqref{eq64.4}, \eqref{eq9.4} and Lemma \ref{le20.4} we end up with
\begin{equation*}
\left\lvert\int_{-1}^1\partial_t u_\delta\ \psi\ dx\right\lvert \leq \left(\delta+1+r\right)\ C(T)\ ||\psi||_{W^{1,2}}\leq C(T)\ ||\psi||_{W^{1,2}}
\end{equation*}
since $0<\delta<1$. A duality argument gives 
\begin{equation*}
||\partial_t u_\delta(t)||_{(W^{1,2})'}\leq  C_6(T),\ \ t\in [0,T]
\end{equation*}
and the proof of Lemma \ref{le4.4} is complete.
\end{proof}
\subsection{Convergence}
In this section we discuss the limit of $(u_\delta, \varphi_\delta)$ as $\delta\rightarrow 0$. For that purpose, we study the compactness properties of $(u_\delta, \varphi_\delta)$.
\begin{proof}[Proof of Theorem \ref{th1.4}]
Thanks to Lemma \ref{le10.4} and \eqref{u4}, $(u_\delta)_\delta$ is bounded in\\ $L^\infty\left((0,T); W^{1,2}(-1,1)\right)$ while $(\partial_t u_\delta)_\delta$ is bounded in $L^\infty\left( (0,T); (W^{1,2})'(-1,1)\right)$ by Lemma~\ref{le4.4}. Since $W^{1,2}(-1,1)$ is compactly embedded in $C[-1,1]$ and $C[-1,1]$ is continuously embedded in $(W^{1,2})'(-1,1)$, it follows from \cite[Corollary 4]{compact} that $(u_\delta)_\delta$ is relatively compact in  $C\left([0,T]\times[-1,1]\right)$. Therefore, there are a sequence $(\delta_j)$ of positive real numbers, $\delta_j\rightarrow 0$, and 
$u\in L^\infty\left((0,T); W^{1,2}(-1,1)\right)$ such that 
\begin{equation}\label{eq11.4}u_{\delta_j}\rightharpoonup u\ \ \mathrm{in}\ L^2\left((0,T); W^{1,2}(-1,1)\right), \end{equation}
and
\begin{equation}\label{eq12.4}
u_{\delta_j}\longrightarrow u\ \ \mathrm{in}\ C\left([0,T]\times [-1,1]\right).
\end{equation}
Owing to Lemma \ref{le0.4} and Lemma \ref{le10.4}, we may also assume that
\begin{equation}\label{eq13.4}
\varphi_{\delta_j}\rightharpoonup\varphi \ \mathrm{in}\ L^2\left((0,T); W^{1,2}(-1,1)\right)\ \mathrm{as}\ \delta_j\rightarrow 0,
\end{equation}
and
\begin{equation}\label{eq14.4}
\delta_j\ \partial_xu_{\delta_j}\longrightarrow  0 \ \mathrm{in} \ L^2\left((0,T)\times (-1,1)\right)\ \mathrm{as}\ \delta_j\rightarrow 0.
\end{equation}
Owing to \eqref{eq11.4}-\eqref{eq14.4}, it is straightforward to deduce from \eqref{bc4} that $(u, \varphi)$ satisfies
\begin{equation}\label{eq35.4}\int_0^T \langle\partial_t u, \psi\rangle \ dt=\int_0^T\int_{-1}^1 \left(u\ \varphi\ \partial_x \psi+r\ u\ (1-u)(u-a)\ \psi\right)\ dx\ dt,\end{equation}
and
\begin{equation}\label{eq36.4}\varepsilon\ \int_0^T\int_{-1}^1 \partial_x \varphi\ \partial_x\psi\  dxdt+\int_0^T\int_{-1}^1\varphi\ \psi\ dxdt=\int_0^T\int_{-1}^1 (-2\ u+a+1)\ \partial_xu\ \psi\  dxdt.\end{equation}
for all test functions $\psi\in C^2\left([0,T]\times (-1,1)\right)$. Since $(u,\varphi)$ satisfies \eqref{eq35.4} and \eqref{eq36.4}, then $(u,\varphi)$ is a weak solution of \eqref{we1.4}, \eqref{we3.4} and \eqref{we4.4}. Recalling that $u\in L^\infty\left((0,T); W^{1,2}(-1,1)\right)$ and $\varphi\in L^\infty\left((0,T); W^{1,2}(-1,1)\right)$ by Lemma \ref{le3.4} and Lemma \ref{le10.4}, we deduce from \eqref{eq35.4} that $\partial_tu \in L^2\left((0,T)\times (-1,1)\right)$ and from \eqref{eq36.4} that $\varphi\in L^\infty\left((0,T); W^{2,2}(-1,1)\right) $, so that $u$ solves \eqref{we1.4} and \eqref{we4.4} in the sense of Definition \ref{de2.4} with the regularity \eqref{eq53.4}. By Proposition~\ref{pr1.4}, such a solution is unique so that $u$ is the only possible cluster point of $(u_\delta)_\delta$ in $C\left([0,T]\times[-1,1]\right)$. Therefore, the whole family $(u_\delta)_\delta$ converges to $u$ in $C\left([0,T]\times[-1,1]\right)$ as $\delta\longrightarrow 0$.
\end{proof}
\section{The monostable case, $E(u)=1-u$}
Let $T>0$, system \eqref{in5.4} now reads
\begin{equation}
\label{mc4}
\left\{
\begin{array}{llll}
\displaystyle \partial_t u_\delta&=&\delta\ \partial^2_{x} u_\delta-\partial_x(u_\delta\ \varphi_\delta)+r\ u_\delta\ (1-u_\delta)& x\in (-1,1),\ t>0 \\
\displaystyle -\varepsilon\ \partial^2_{x} \varphi_\delta+\varphi_\delta&=&-\ \partial_x u_\delta,& x\in (-1,1),\ t>0 \\
\displaystyle \partial_xu_\delta(t,\pm1)=\varphi_\delta(t,\pm1)&=&0& t>0,\\
\displaystyle u_\delta(0,x)&=&u_0(x)& x\in (-1,1),
\end{array}
\right.
\end{equation}
Thanks to Theorem \ref{th3.4}, system \eqref{mc4} has a unique global nonnegative solution in the sense of Definition \ref{de1.4}.\\
Unlike the previous case, it does not seem to be possible to begin the proof with an $L^\infty(L^2)$ estimate on $u_\delta$. Nevertheless, there is still a cancellation between the two equations which actually gives an $L^\infty(L \log L)$ bound on $u_\delta$ and a $L^2$ bound on $\partial_x\sqrt u_\delta$ as we shall see below. Integrating \eqref{mc4} over $[0,T]\times(-1,1)$ and using the nonnegativity of $u_\delta$, we first observe that,
\begin{equation}\label{eq15.4}
||u_\delta(t)||_1\leq ||u_0||_1+2\ r\ t,\ \ \mathrm{for\ all}\ t\in [0, T].
\end{equation}
\subsection{Estimates}
\begin{lemma}\label{le5.4}
There is $C_7(T)>0$ independent of $\delta$ such that
\begin{equation}\label{eq19.4}
\int_0^T \left(\varepsilon\ ||\partial_x\varphi_\delta||_2^2+||\varphi_\delta||_2^2+4\ \delta\ ||\partial_x \sqrt{u_\delta}||_2^2\right)\ dt\leq C_7(T),\ \ \mathrm{for\ all}\ t\in [0, T],
\end{equation}
\begin{equation}\label{eq38.4}
\int_0^T\ ||\varphi_\delta||^2_\infty\ dt\leq C_7(T),\ \ \mathrm{for\ all}\ t\in [0, T].
\end{equation}
\end{lemma}
\begin{proof}
The proof goes as follows. On the one hand, we multiply the first equation in \eqref{mc4} by $(\log u_\delta+1)$ and integrate it over $(-1,1)$. Since $u_\delta\ (1-u_\delta)\ \log u_\delta\leq 0$ and $u_\delta\ (1-u_\delta)\leq 1$,
\begin{eqnarray}
\frac{\mathrm{d}}{\mathrm{d}t}\int_{-1}^1 u_\delta\ \log u_\delta\ dx&=& -\int_{-1}^1 (\delta\ \partial_x u_\delta-u_\delta\ \varphi_\delta)\ (\frac{1}{u_\delta}\ \partial_x u_\delta)\ dx\nonumber\\
&+&r\ \int_{-1}^1 u_\delta\ (1-u_\delta)\ (\log u_\delta+1)\ dx\nonumber\\
&\leq& -\int_{-1}^1 \frac{\delta}{u_\delta}\ (\partial_x u_\delta)^2\ dx+\int_{-1}^1\varphi_\delta\ \partial_xu_\delta\ dx+2\ r.\label{eq40.4}
\end{eqnarray}
On the other hand, we multiply the second equation in \eqref{mc4} by $\varphi_\delta$ and integrate it over $(-1,1)$ to obtain
\begin{equation}\label{eq17.4}\varepsilon \int_{-1}^1 |\partial_x \varphi_\delta|^2\ dx+\int_{-1}^1 |\varphi_\delta|^2\ dx=-\int_{-1}^1 \partial_x u_\delta\ \varphi_\delta\ dx.\end{equation}
Adding \eqref{eq40.4} and \eqref{eq17.4} yields
\begin{equation}\label{eq18.4}
\frac{\mathrm{d}}{\mathrm{d}t}\int_{-1}^1 u_\delta\ \log u_\delta\ dx+\varepsilon\ ||\partial_x\varphi_\delta||_2^2+||\varphi_\delta||_2^2\leq -4\ \delta\int_{-1}^1 |\partial_x \sqrt{u_\delta}|^2\ dx+2\ r.
\end{equation}
Then, \eqref{eq19.4} is obtained by a time integration of \eqref{eq18.4}. Finally, by the continuous embedding of $W^{1,2}(-1,1)$ in $L^\infty(-1,1)$ we obtain \eqref{eq38.4}.
\end{proof}
\begin{lemma}\label{le6.4}
For $0<\delta<1$, there exists $C_8(T)>0$ independent of $\delta$ such that
\begin{equation}
\partial_x \varphi_\delta (x)\geq -4\ ||\varphi_\delta||_\infty-C_8(T)\ \ \mathrm{for\ all}\ t\in [0,T].
\end{equation}
\end{lemma}
\begin{proof}
For $x,y \in (-1,1)$, we integrate the second equation in \eqref{mc4} to obtain
\begin{eqnarray}
-\varepsilon\ \int_y^x \partial^2_{x}\varphi_\delta(z)\ dz+\int_y^x \varphi_\delta(z)\ dz&=&-\int_y^x \partial_xu_\delta\ dz\nonumber\\
-\varepsilon\ \left(\partial_x \varphi_\delta(x)-\partial_x \varphi_\delta(y)\right)+\int_y^x \varphi_\delta(z)\ dz&=&-[u_\delta(x)-u_\delta(y)].\nonumber
\end{eqnarray}
Next we integrate the above equality with respect to $y$ over $(-1,1)$ to obtain
\begin{equation*}
-2\ \varepsilon\ \partial_x \varphi_\delta(x)=-\int_{-1}^1\int_y^x \varphi_\delta(z)\ dzdy-2\ u_\delta(x)+||u_\delta||_1.
\end{equation*}
Since 
\begin{eqnarray*}
\int_{-1}^1\int_y^x \varphi_\delta(z)\ dzdy
&\geq& -4\ ||\varphi_\delta||_\infty,
\end{eqnarray*}
and  $ u_\delta\geq 0,$ it follows from \eqref{eq15.4} that
\begin{eqnarray}
\partial_x \varphi_\delta(x)&=&\frac{1}{2\ \varepsilon}\int_{-1}^1\int_y^x \varphi_\delta(z)\ dzdy+\frac{1}{\varepsilon}\ u_\delta(x)-\frac{1}{2\ \varepsilon}\ ||u_\delta||_1\nonumber\\
&\geq& -\frac{2}{\varepsilon} \ ||\varphi_\delta||_\infty-C_8(T).\nonumber
\end{eqnarray}
\end{proof}
\begin{lemma}\label{le21.4}
There is $C_9(T)>0$ independent of $\delta$ such that
\begin{equation}\label{eq71.4}
||\varphi_\delta(t)||_2\leq C_9(T),\ \ \mathrm{for\ all}\ t\in [0,T].
\end{equation}
\end{lemma}
\begin{proof}
The proof is similar to that of Lemma \ref{le20.4}.
\end{proof}
Now, we continue with estimates for the derivatives of $u_\delta$.
\begin{lemma}\label{le7.4}
There is $C_{10}(T)>0$ independent of $\delta$ such that
\begin{equation}\label{eq20.4}
||\partial_x u_\delta(t)||_1\leq C_{10}(T)\ \ \mathrm{for\ all}\ t\in [0,T].
\end{equation}
\end{lemma}
\begin{proof}
Differentiating the first equation in \eqref{mc4} with respect to $x$ and setting $g_\delta=\partial_x u_\delta$ yield
\begin{equation}\label{eq21.4}
\partial_t (g_\delta)+\partial_{x}^2 (u_\delta\ \varphi_\delta)-r\ \partial_x\left(u_\delta\ (1-u_\delta)\right)=\delta\ \partial^2_{x} g_\delta.
\end{equation}
We define as in the bistable case an approximation of the sign function by $\sigma_\gamma(z)=\sigma(\frac{z}{\gamma})$, $0<\gamma\ll 1$, with $\sigma$ smooth and increasing, $\sigma(0)=0$, and $\sigma(z)=\mathrm{sign}\ z$ for $|z|>1$. Then, with $\mathrm{abs}_\gamma(z)=\int_0^z\sigma_\gamma(\xi)\ d\xi$, the convergence of $\mathrm{abs}_\gamma(z)$ to $|z|$ as $\gamma\rightarrow 0$ is uniform in $z\in \mathbb{R}$.\\
Multiplying \eqref{eq21.4} by $\sigma_\gamma(g_\delta)$ and integrating with respect to $x$ yields
\begin{eqnarray}
\int_{-1}^1 \sigma_\gamma(g_\delta)\ \partial_t (g_\delta)\ dx&+&\int_{-1}^1\sigma_\gamma(g_\delta)\ [\partial_{x}(g_\delta\ \varphi_\delta)+g_\delta\ \partial_x\varphi_\delta+u_\delta\ \partial^2_{x}\varphi_\delta]\ dx\nonumber\\
&-&r\int_{-1}^1 \sigma_\gamma(g_\delta)\ (1-2\ u_\delta)\ g_\delta\ dx\nonumber\\
&=&\delta\ \int_{-1}^1 \sigma_\gamma(g_\delta)\ \partial_{x}^2 g_\delta\ dx.\label{eq22.4}
\end{eqnarray}
Since
\begin{equation*}
\int_{-1}^1\sigma_\gamma(g_\delta)\ \partial_{x}( g_\delta\ \varphi_\delta)\ dx=-\int_{-1}^1\sigma'_\gamma(g_\delta)\ \partial_xg_\delta\ g_\delta\ \varphi_\delta\ dx
\end{equation*}
and
\begin{equation*}
\delta\ \int_{-1}^1 \sigma_\gamma(g_\delta)\ \partial_{x}^2 g_\delta\ dx=-\delta\ \int_{-1}^1 (\partial_xg_\delta)^2\ \sigma'_\gamma(g_\delta)\ dx
\end{equation*}
 we obtain
\begin{eqnarray}
\frac{\mathrm{d}}{\mathrm{d}t}\int_{-1}^1 \mathrm{abs}_\gamma(g_\delta)\ dx&-&\int_{-1}^1\sigma'_\gamma(g_\delta)\ \partial_xg_\delta\ g_\delta\ \varphi_\delta\ dx\nonumber\\
&+&\int_{-1}^1\sigma_\gamma(g_\delta)\ g_\delta\ \partial_x\varphi_\delta\ dx+\int_{-1}^1\sigma_\gamma(g_\delta)\ u_\delta\ \partial^2_{x}\varphi_\delta\ dx\nonumber\\
&-&r\ \int_{-1}^1 \sigma_\gamma(g_\delta)\ (1-2\ u_\delta)\ g_\delta\ dx\nonumber\\
&=&-\delta\ \int_{-1}^1 (\partial_xg_\delta)^2\ \sigma'_\gamma(g_\delta)\ dx\leq 0.\label{eq23.4}
\end{eqnarray}
The function $f_\gamma(z)=\sigma_\gamma(z)\ z-\mathrm{abs}_\gamma (z)$ satisfies $f'_\gamma(z)=\sigma'_\gamma(z)\ z$ and converges to $0$ uniformly in $z\in \mathbb{R}$. We integrate the second term in \eqref{eq23.4} by parts and we use the second equation in \eqref{mc4} in the fourth one to obtain
\begin{eqnarray}
\frac{\mathrm{d}}{\mathrm{d}t}\int_{-1}^1 \mathrm{abs}_\gamma(g_\delta)\ dx&+&\int_{-1}^1 f_\gamma(g_\delta)\ \partial_x\varphi_\delta\ dx+\int_{-1}^1\sigma_\gamma(g_\delta)\ g_\delta\ \partial_x\varphi_\delta\ dx\nonumber\\
&\leq&-\frac{1}{\varepsilon}\int_{-1}^1\sigma_\gamma(g_\delta)\ u_\delta\  g_\delta\ dx-\frac{1}{\varepsilon}\int_{-1}^1\sigma_\gamma(g_\delta)\ u_\delta\ \varphi_\delta\  dx\nonumber\\
&+&r\int_{-1}^1 \sigma_\gamma(g_\delta)\ (1-2\  u_\delta)\ g_\delta\ dx.\nonumber\\
&\leq&\frac{1}{\varepsilon} \ ||u_\delta||_1\ ||\varphi_\delta||_\infty+r\ \int_{-1}^1 |g_\delta|\ dx\nonumber,
\end{eqnarray}
since $u_\delta\geq 0$, $1-2u_\delta\leq1$ and $0\leq \sigma_\gamma(g_\delta)\ g_\delta\leq |g_\delta|$.
Passing to the limit $\gamma\rightarrow 0$ in the above inequality, the second term on the left-hand side vanishes. It follows from Lemma~\ref{le6.4} and \eqref{eq15.4} that
\begin{eqnarray}
\frac{\mathrm{d}}{\mathrm{d}t}\int_{-1}^1 |g_\delta|\ dx&-&\left(\frac{2}{\varepsilon}\ ||\varphi_\delta||_\infty+C_8(T)\right)\ \int_{-1}^1 |g_\delta|\ dx\nonumber\\
&\leq&\frac{C(T)}{\varepsilon}\ ||\varphi_\delta||_\infty+r\ \int_{-1}^1 |g_\delta|\ dx.\label{eq25.4}
\end{eqnarray}
Integrating \eqref{eq25.4} in time, and using \eqref{eq38.4} yield that there exists $C_{10}(T)$ such that \eqref{eq20.4} holds.
\end{proof}
As in the previous section, we have the following consequence of Lemma \ref{le21.4} and Lemma~\ref{le7.4}.
\begin{lemma}\label{le8.4}
There is $C_{11}(T)>0$ independent of $\delta$ such that
\begin{equation}\label{eq39.4}
||\partial_x\varphi_\delta(t)||_\infty+||u_\delta(t)||_\infty\leq C_{11}(T)\ \ \mathrm{for\ all}\ t\in [0,T].
\end{equation}
\end{lemma}
\begin{lemma}\label{le12.4}
There is $C_{12}(T)>0$ independent of $\delta$ such that
\begin{equation}\label{eq68.4}
||\partial_xu_\delta(t)||_2\leq C_{12}(T)\ \ \mathrm{for\ all}\ t\in [0,T].
\end{equation}
\begin{proof}
Since \eqref{eq39.4} holds, we argue as in the proof of Lemma \ref{le10.4} to obtain \eqref{eq68.4}.
\end{proof}
\end{lemma}
\begin{lemma}\label{le9.4}
There is $C_{13}(T)>0$ independent of $\delta$ such that
\begin{equation}\label{eq10.4}
 ||\partial_t u_\delta(t)||^2_{(W^{1,2})'}\leq C_{13}(T)\ \ \mathrm{for\ all}\ t\in [0,T].
\end{equation}
\end{lemma}
\begin{proof}
Consider $\psi \in W^{1,2}(-1,1)$ and $t\in (0, T)$. We have by the first equation in \eqref{mc4}
\begin{eqnarray}
&&\left\lvert\int_{-1}^1 \partial_t u_\delta\ \psi\ dx\right\lvert\nonumber\\
&=& \left\lvert\int_{-1}^1 \left( -\delta\ \partial_x u_\delta\ \partial_x\psi+u_\delta\ \varphi_\delta\ \partial_x\psi +r\ u_\delta\ (1-u_\delta)\ \psi\right)\ dx\right\lvert\nonumber\\
&\leq& \delta\ ||\partial_x\psi||_2\ ||\partial_x u_\delta||_2+||\partial_x\psi||_2\ \ ||u_\delta||_\infty\ ||\varphi_\delta||_2+r\ ||u_\delta\ (1-u_\delta)||_2\ ||\psi||_2\nonumber.
\end{eqnarray} 
Using Lemma \ref{le21.4}, Lemma \ref{le8.4} and Lemma \ref{le12.4}, we end up with
\begin{eqnarray}
\left\lvert\int_{-1}^1\partial_t u_\delta\ \psi\ dx\right\lvert &\leq & C(T)\ ||\psi||_{W^{1,2}}\nonumber,
\end{eqnarray}
and a duality argument gives 
\begin{equation*}
||\partial_t u_\delta(t)||_{(W^{1,2})'}\leq C(T),\ \ t\in[0,T]
\end{equation*}
and the proof of Lemma \ref{le9.4} is complete.
\end{proof}
\subsection{Convergence}
\begin{proof}[Proof of Theorem \ref{th2.4}]
Thanks to the previous analysis, the proof of Theorem \ref{th2.4} can then be done as that of Theorem \ref{th1.4}. 
\end{proof}
\section*{Acknowledgment}
I thank Philippe Lauren\c cot for valuable and fruitful discussions.

\end{document}